\documentclass[a4,12pt]{amsart}
\usepackage{amssymb,amstext,amsmath,amscd,amsthm,amsfonts,enumerate,graphicx,latexsym}
\usepackage[usenames]{color}
\newtheorem{thm}{Theorem}[section]

\newtheorem{lem}[thm]{Lemma}
\newtheorem{prop}[thm]{Proposition}
\theoremstyle{definition}
\newtheorem{dfn}[thm]{Definition}

\newtheorem{ques}[thm]{Question}

\newtheorem{claim}{Claim}
\newtheorem*{claim*}{Claim}
\theoremstyle{remark}
\newtheorem*{ac}{Acknowlegments}
\renewcommand{\qedsymbol}{$\blacksquare$}
\numberwithin{equation}{thm}

\def\Ext{\operatorname{\mathsf{Ext}}}

\def\syz{\mathsf{\Omega}}

\def\depth{\operatorname{\mathsf{depth}}}

\def\pd{\operatorname{\mathsf{pd}}}

\def\Hom{\operatorname{\mathsf{Hom}}}
\def\depth{\operatorname{\mathsf{depth}}}
\def\length{\operatorname{\mathsf{length}}}
\def\rank{\operatorname{\mathsf{rank}}}
\def\End{\operatorname{\mathsf{End}}}
\def\Cok{\operatorname{\mathsf{Coker}}}
\def\Ker{\operatorname{\mathsf{Ker}}}
\def\edim{\operatorname{\mathsf{edim}}}
\begin{document}
\allowdisplaybreaks
\setlength{\baselineskip}{15pt}
\title[Ranks of syzygies of modules of finite length]{On a question of Buchweitz about ranks of syzygies of modules of finite length}
\author{Toshinori Kobayashi}
\address{Graduate School of Mathematics, Nagoya University, Furocho, Chikusaku, Nagoya, Aichi 464-8602, Japan}
\email{m16021z@math.nagoya-u.ac.jp}
\begin{abstract}
Let $R$ be a local ring of dimension $d$. Buchweitz asks if the rank of the $d$-th syzygy of a module of finite lengh is greater than or equal to the rank of the $d$-th syzygy of the residue field, unless the module has finite projective dimension. Assuming that $R$ is Gorenstein, we prove that if the question is affrmative, then $R$ is a hypersurface. If moreover $R$ has dimension two, then we show that the converse also holds true.
\end{abstract}
\maketitle
\section{Introduction}
Let $(R,\mathfrak{m},k)$ be a commutative Noetherian local ring with Krull dimension $d$. We consider the rank of the $d$-th syzygy of an $R$-module of finite length. We assume that $R$ has positive depth, so that any $R$-module of finite lengh has a rank. On the ranks of syzygies, Buchweitz asks the following question \cite[Question 11.16]{LW}.

\begin{ques}[Buchweitz] \label{Q}

Does one have the equality 
\begin{equation} \label{Q1}
\rank_R \syz^d k =\min\{\rank_R \syz^d M|\pd_R M=\infty\text{ and }\length_R M<\infty\}?
\end{equation}
\end{ques}

Here we denote by $\syz^d M$ the $d$-th syzygy in the minimal free resolution of a finitely generated $R$-module $M$, and $\pd_R M$ stands for the projective dimension of $M$. If $d=1$, then $\syz^d k$ has rank one, and Question \ref{Q} has an affirmative answer. Therefore, we consider the question for $d\geq 2$. Our main theorem is the following.

\begin{thm} \label{A}
Assume $R$ is Gorenstein and $d\geq 2$. Then Question \ref{Q} is affirmative only if $R$ is a hypersurface.
\end{thm}

Here we say that $R$ is a {\it hypersurface} if the $\mathfrak{m}$-adic completion of $R$ is a quotient of a regular local ring by a regular element. This theorem says that if $R$ is a Gorenstein local ring and not a hypersurface, then Question \ref{Q} has a negative answer.

On the other hand we can show the converse of Theorem \ref{A} in the case $d=2$.

\begin{thm} \label{B}
Assume $R$ is Gorenstein and $d=2$. Then Question \ref{Q} is affirmative if and only if $R$ is a hypersurface.
\end{thm}

This paper is organized as follows. In Section 2, we give a necessary condition for the equality (\ref{Q1}) over a Gorenstein ring. In Section 3, we consider the Poincar\'e series of $k$, and prove Theorem \ref{A} by using the necessary condition obtained in Section 2. Section 4 is devoted to proving Theorem \ref{B} by using the notion of Buchsbaum-Rim complexes.

\section{A necessary condition for (\ref{Q1})}

Throughout this section, $(R,\mathfrak{m},k)$ is a Gorenstein local ring of dimension $d>0$. To prove Theorem \ref{A}, we use the following result which provides a necessary condition for the equality (\ref{Q1}) to hold true.

\begin{prop} \label{C}
There is an $R$-module $M$ with $\pd_R M=\infty$, $\length_R M<\infty$, and $\rank_R\syz^d M=\rank_R\syz^{d-1} k+1$. Thus if Question \ref{Q} is affirmative, then there is an inequality
\begin{equation} \label{C0}
\rank_R \syz^d k\leq \rank_R \syz^{d-1} k+1.
\end{equation}
\end{prop}

In the rest of this section, we prove this proposition. First, we state the definition of a (minimal) MCM approximation.

\begin{dfn} (see \cite[Chap. 11, Section 2]{LW})
For a finitely generated $R$-module $M$, an {\it MCM approximation} of $M$ is a pair $(X,p)$ of a maximal Cohen-Macaulay $R$-module $X$ and a surjective homomorphism $p:X\rightarrow M$ with $\pd_R (\Ker p)<\infty$.
An MCM approximation $(X,p)$ of $M$ is called {\it minimal} if every $\phi\in\End_R(X)$ with $p\circ\phi=p$ is an automorphism.
\end{dfn}

Since $R$ is Gorenstein, an (minimal) MCM approximaion exists for any finitely generated $R$-module. We remark that an MCM approximaion of $M$ is unique up to free summands, and a minimal MCM approximation of $M$ is unique up to isomorphism. We denote by $X_M$ the maximal Cohen-Macaulay $R$-module in the minimal MCM approximation of $M$.

For an $R$-module $M$ of finite length, we can construct $X_M$ from the Matlis dual of $M$ as follows (see the proof of \cite[Proposition 11.15]{LW}).

\begin{lem} \label{C1}
Let $M$ be an $R$-module of finite length. Then $X_M\cong\Hom_R(\syz^d \Ext^d_R(M,R),R)$. In particular, $\rank_R X_M=\rank_R \syz^d \Ext^d_R(M,R).$
\end{lem}

The rank of the minimal MCM approximation of $\syz^{d-1} k$ is computed from that of $\syz^{d-1} k$.

\begin{lem} \label{C4}
One has $\rank_R X_{\syz^{d-1}k}=\rank_R \syz^{d-1}k+1$.
\end{lem}

\begin{proof}
Since $M:=\syz^{d-1} k$ has depth $d-1$, we have a short exact sequence
\[
0\to R^{\oplus r}\to X_M \to M \to 0
\]
and $r=\length_R\Ext^1_R(M,R)=\length_R\Ext^d_R(k,R)=1$ by \cite[Proposition 11.21]{LW}. Thus we have $\rank_R X_M=\rank_R M+1$.
\end{proof}

The rank of a maximal free summand of $X_M$ is called the{\it (Auslander) delta invariant} of $M$ and denoted by $\delta_R(M)$. We note that $\delta_R(M)$ is well-defined without the Krull-Schmidt property of finitely generated $R$-modules. We give some properties of delta invariants in the next lemma.

\begin{lem} \label{C2}
Let $M,N$ be finitely generated $R$-modules. The following hold.
\begin{itemize}
\item[(1)] If there exists a surjective homomorphism $M\rightarrow N$, then $\delta_R(M)\geq \delta_R(N)$.
\item[(2)] If $R$ is not regular, then $\delta_R(\syz^i k)=0$ for all $i\geq 0$.
\end{itemize}
\end{lem}

\begin{proof}
See \cite[Proposition 11.28]{LW} and \cite[Proposition 5.7]{AR} respectively.
\end{proof}

The following proposition plays a key role in the proof of Proposition \ref{C}.
\begin{prop} \label{C3}
There is an $R$-module $M$ with $\pd_R M=\infty$, $\length_R M<\infty$, and $X_M\cong X_{\syz^{d-1} k}$.
\end{prop}

\begin{proof}
Let $\underline{x}=x_1,\dots,x_d$ be a system of parameters of $R$ with $x_i\not\in \mathfrak{m}^2+(x_1,\dots,x_{i-1})$ for all $i$, and set $R'=R/(\underline{x})$. Put $M$ to be the $R$-module $\syz^{d-1}_{R'} k$. Then $\length_R M<\infty$ and $\pd_R M=\infty$ since  $\length_R R'<\infty$, $\pd_R R'<\infty$, and $\pd_R k=\infty$. We want to show that $X_M\cong X_{\syz^{d-1}_Rk}$. To prove this, it is enough to show that the following two claims hold.
\begin{claim}
Let $0\to A \to B \to C \to 0$ be an exact sequence of $R$-modules with $\pd_R B<\infty$. Then $X_A\cong X_{\syz C}\cong \syz X_C$ up to free summands. Consequently, $X_M\cong X_{\syz^{d-1}_R k}$ up to free summands.
\end{claim}
\begin{claim}
One has $\delta_R(M)=0=\delta_R(\syz^{d-1}_Rk)$.
\end{claim}
\begin{proof}[Proof of Claim 1]
There is an exact sequence $0 \to \syz B \to \syz C\oplus P \to A \to 0$ with some free module $P$. Let $W$ be a pull-back of $\syz B \to \syz C\oplus P$ and $p:X_{\syz C}\rightarrow \syz C$. Then the induced sequences $0 \to \Ker p \to W \to B \to 0$ and $0 \to W \to X_{\syz C} \to A \to 0$ are both exact. As $\pd_R B<\infty$ and $\pd_R (\Ker p)<\infty$, we have $\pd_R W<\infty$. So $X_{\syz C}$ is an MCM approximation of $A$ and isomorphic to $X_A$ up to free summands. Applying this argument repeatedly, we get $X_M \cong X_{\syz_R(\syz^{d-2}_{R'}k)}\cong \syz_R X_{\syz^{d-2}_{R'}k}\cong \syz^2_R X_{\syz^{d-3}_{R'}k}\cong\cdots\cong\syz^{d-1}_RX_k\cong X_{\syz^{d-1}_Rk}$ up to free summands.
\renewcommand{\qedsymbol}{$\square$}
\end{proof}
\begin{proof}[Proof of Claim 2]
By Lemma \ref{C2}, it is enough to show that there is an epimorphism $\syz^{d-1}_R k \to M$. We show this by induction on $d$. The case $d=1$ is trivial. So we assume $d>1$. Put $S=R/(x_1)$. Since $x_1\in \mathfrak{m}\setminus \mathfrak{m}^2$, one obtains 
$\syz^1_R k\otimes_R S\cong k\oplus \syz^1_S k$. So $\syz^{d-1}_R k\otimes_R S\cong \syz^{d-2}_Sk\oplus \syz^{d-1}_S k$. In particular, there is an epimorphism $\syz^{d-1}_R k \to \syz^{d-1}_S k$. By the hypothesis of induction, there is an epimorphism $\syz^{d-1}_S k \to M$. So we have an epimorphism $\syz^{d-1}_R k \to M$.
\renewcommand{\qedsymbol}{$\square$}
\end{proof}
The proof of the proposition is thus completed.
\end{proof}

Now we can give a proof of Proposition \ref{C}.

\begin{proof}[Proof of Proposition \ref{C}]
By Proposition \ref{C3} and Lemma \ref{C4}, there exists an $R$-module $M$ with $\pd_R M=\infty$, $\length_R M<\infty$, and $\rank_R X_M=\rank_R \syz^{d-1}k+1$. On the other hand, $\rank_R X_M=\rank_R \syz^d\Ext^d_R(M,R)$ by Lemma \ref{C3}. Since $M$ has finite length, $M'=\Ext^d(M,R)$ also has finite length. Since $\pd_R M=\infty$ and $R$ is Gorenstein, we see that $M'$ has infinite projective dimension. So $M'$ satisfies the condition of Proposition \ref{C}, that is, $\pd_R M'=\infty$, $\length_R M'<\infty$, and $\rank_R \syz^d M'(=\rank_R X_M)=\rank_R \syz^{d-1}k+1$. 
\end{proof}

\section{The Poincar\'e series of the residue field}

Throughout this section, $(R,\mathfrak{m},k)$ is a local ring with $\depth R>0$. So any $R$-module of finite length has rank $0$. For a finitely generated $R$-module $M$, we denote by $\beta_i(M)$ the $i$-th Betti number of $M$. Then the formal power series $P_M(t):=\sum_{i=0}^\infty \beta_i(M)t^i$ is called the {\it Poincar\'e series} of $M$. Since $\beta_0(k)=1$, there are integers $\varepsilon_i$ and an equality
\[
P_k(t)=\frac{\prod_{i=1}^\infty (1+t^{2i-1})^{\varepsilon_{2i-1}}}{\prod_{j=1}^\infty(1-t^{2j})^{\varepsilon_{2j}}};
\]

\noindent see \cite[Remark 7.1.1]{A}. For example, it holds that $\varepsilon_1=\beta_1(k)=\edim R$ and $\varepsilon_2=\beta_2(k)-\binom{\beta_1(k)}{2}$, where $\edim R$ stands for the embedding dimension of $R$. By \cite[Corollary 7.1.4]{A}, we have $\varepsilon_i \geq 0$ for all $i$. So there is a formal power series $Q(t)\in \mathbb{Z}[[t]]$ with non-negative coefficients and $Q(0)=1$ such that 
\[
P_k(t)=\frac{(1+t)^{\varepsilon_1}}{(1-t^2)^{\varepsilon_2}}Q(t).
\]
The equality $\rank_R \syz^i k +\rank_R \syz^{i+1} k=\beta_i(k)$ yields
\[
\sum_{i=1}^\infty (\rank_R \syz^i k)t^i=\frac{t}{1+t}P_k(t).
\]
So we have
\begin{align*}
\sum_{i=1}^\infty (\rank_R \syz^i k-\rank_R \syz^{i-1} k)t^i & =(1-t)\sum_{i=1}^\infty (\rank_R \syz^i k)t^i\tag{3.0.1}\\
& =t\frac{1-t}{1+t}P_k(t)=t\frac{(1+t)^{\varepsilon_1-2}}{(1-t^2)^{\varepsilon_2-1}}Q(t).
\end{align*}

From this equation, the main proposition of this section is deduced.

\begin{prop} \label{D}
The inequality
\[
\rank_R \syz^d k\leq \rank_R \syz^{d-1}k+1
\]
implies that $R$ is a hypersurface or that $d=1$.
\end{prop}

\begin{proof}
Since completion does not change the Betti numbers of $k$, we may assume that $R$ is complete. Then $R$ admits a presentation $R=S/I$ with a regular local ring $(S,\mathfrak{n})$ and an ideal $I\subset \mathfrak{n}^2$ of $S$. By \cite[Theorem 2.3.2]{BH}, the number $\varepsilon_2=\beta_2(k)-\binom{\mathsf{ediim}R}{2}$ is equal to $\beta_0^S(I)$. Now we assume that $R$ is not a hypersurface and $d\geq 2$. Therefore one has $\varepsilon_1=\edim R\geq d+2$ and $\varepsilon_2 \geq 2$. The formal power series $Q'(t)=\frac{1}{(1-t^2)^{\varepsilon_2-1}}Q(t)$ also has non-negative coefficients and satisfies $Q'(0)=1$, because $\varepsilon_2\geq 1$. As a consequence of these observations, we see that
\begin{align*}
\text{the $d$-th coefficient of }t\frac{(1+t)^{\varepsilon_1-2}}{(1-t^2)^{\varepsilon_2-1}}Q(t) &=\text{the $(d-1)$-th coefficient of }(1+t)^{\varepsilon_1-2}Q'(t)\\
&\geq \binom{\varepsilon_1-2}{d-1}\geq \binom{d}{d-1}\geq d\geq 2.
\end{align*}
Combining this with the equation (3.0.1), we obtain $\rank_R\syz^dk-\rank_R\syz^{d-1}k\geq 2$.
\end{proof}

Now we can easily see that Proposition \ref{C} and \ref{D} implies Theorem \ref{A}.

\begin{proof}[Proof of Theorem \ref{A}]
Assume that Question \ref{Q} has an affirmative answer. Proposition \ref{C} yields that the inequality $\rank_R \syz^d k\leq \rank_R \syz^{d-1} k+1$ holds. Then the ring $R$ needs to be a hypersurface because of the consequence of Proposition \ref{D}.
\end{proof}

\section{The case of dimension two}

The aim of this section is to prove Theorem \ref{B}. In this section, $(R,\mathfrak{m},k)$ is a Cohen-Macaulay local ring of dimension $d>0$. Let $M$ be an $R$-module of finite length and $\phi:R^{\oplus n}\rightarrow R^{\oplus m}$ be a homomorphism of free $R$-modules such that $\Cok \phi=M$. We denote by $I_m(\phi)$ the ideal of $R$ generated by $m$-minors of $\phi$. Taking a non-maximal prime ideal $\mathfrak{p}$ of $R$, we have $M_{\mathfrak{p}}=0$. So $\phi_\mathfrak{p}:R_\mathfrak{p}^{\oplus n}\to R_\mathfrak{p}^{\oplus m}$ is surjective and $n\geq m$. Moreover, $(I_m(\phi))_\mathfrak{p}=I_m(\phi_\mathfrak{p})$ is equal to $R_\mathfrak{p}$. Consequently, $I_m(\phi)$ is an $\mathfrak{m}$-primary ideal of $R$.

To prove Theorem \ref{B}, we want to estimate the rank of $\syz^2 M$. It follows immediately from the exactness of $0\to \syz^2 M \to R^{\oplus n} \to R^{\oplus m} \to M \to 0$ that $\rank_R \syz^2 M=n-m$. We can evaluate the number $n-m$ from the next two propositions.

\begin{prop}\cite[Theorem 3]{EN}. \label{E1}
Let $\phi:R^{\oplus n}\to R^{\oplus m}$ be a homomorphism of free $R$-modules. Then for each integer $0\leq t\leq \min\{m,n\}$, we have
$
\mathsf{ht}~I_t(\phi)\leq (m-t+1)(n-t+1).
$
\end{prop}

\begin{prop}\cite[Corollary 2.7]{BR}. \label{E2}
Let $\phi:R^{\oplus n}\to R^{\oplus m}$ be a homomorphism of free $R$-modules. Assume $n\geq m$. If $\mathsf{grade}~ I_m(\phi)=m-n+1$, then $\pd_R (\Cok \phi)=m-n+1$. 
\end{prop}

Using these proposition, we can give a proof of Theorem \ref{B}.

\begin{proof}[Proof of Theorem \ref{B}]
We recall that the ideal $I_m(\phi)$ is $\mathfrak{m}$-primary. Proposition \ref{E1} yields that $n-m+1\geq \mathsf{ht}~I_m(\phi)\geq \mathsf{grade}~I_m(\phi)=d$. Proposition \ref{E2} says that if $M$ has infinite projective dimension, then $d\not=n-m+1$. So we have $d<n-m+1$. This inequality is equivalent to the inequality $d\leq n-m (=\rank_R \syz^2 M)$. The argument above implies that 
\[
d\leq \min\{\rank_R \syz^2 M|\pd_R M=\infty\text{ and }\length_R M<\infty\}.
\]
Now we assume that $R$ is a hypersurface and $d=2$. Then $\rank \syz^2 k=\beta_1(k)-\beta_0(k)=\edim R-1=2$. The inequality above shows that Question \ref{Q} is affirmative.

The ``only if" part follows from Theorem \ref{A}.
\end{proof}

\begin{ac}
The author is grateful to his supervisor Ryo Takahashi for giving him helpful advice throughout the paper. The author also very much thanks Ragnar Buchweitz for checking the first draft of the paper and giving useful comments.
\end{ac}



\begin{thebibliography}{9}
\bibitem{AR}
{\sc M. Auslander, S. Ding, \O. Solberg}, Liftings and weak liftings of modules, {\em J. Algebra} 156 (1993), 273--317
\bibitem{A}
{\sc L. Avramov}, Infinite free resolutions, Six lectures on commutative algebra, Mod. Birkh¨auser Class., Birkh¨auserVerlag, Basel, 2010, pp. 1–-118.
\bibitem{BH}
{\sc W. Bruns; J. Herzog}, Cohen--Macaulay rings, revised edition, Cambridge Studies in Advanced Mathematics, 39, {\it Cambridge University Press, Cambridge}, 1998.
\bibitem{BR}
{\sc D. A. Buchsbaum, D. S. Rim}, A generalized Koszul Complex II. Depth and multiplicity, {\em Trans. Amer. Math. Sot.} {\bf 3} (1964), 197.
\bibitem{EN}
{\sc J. A. Eagon D. G. Northcott}, Ideals defined by matrices and a certain complex associated with them, {\em Proc. Roy. Soc. London Ser.} A 269 (1962), 188--204.
\bibitem{LW}
{\sc G. J. Leuschke; R. Wiegand}, Cohen-Macaulay Representations, Mathematical Surveys and Monographs, vol. 181, {\em American Mathematical Society, Providence, RI}, 2012.
\end{thebibliography}
\end{document}